\documentclass[preprint,12pt,3p]{elsarticle}

\usepackage{amssymb,amsmath,amsthm,amsfonts}
\usepackage{mathptmx, lineno}
\usepackage[colorlinks,urlcolor=blue,citecolor=blue,linkcolor=blue]{hyperref}

\newtheorem{theorem}{Theorem}[section]

\newtheorem{lemma}[theorem]{Lemma}

\newtheorem{proposition}[theorem]{Proposition}
\theoremstyle{definition}

\theoremstyle{remark}
\newtheorem{remark}[theorem]{Remark}

\numberwithin{equation}{section}
\allowdisplaybreaks

\newcommand{\eps}{\varepsilon}
\newcommand{\R}{\mathbb{R}}

\journal{}

\begin{document}

\begin{frontmatter}

\title{On the Schr\"{o}dinger-Poisson system with indefinite potential and $3$-sublinear nonlinearity\tnoteref{tle}}
\tnotetext[tle]{The first author was supported by NSFC (Grant No.11671331), the second author by project PRIN 2017 ``Nonlinear Differential Problems via Variational, Topological and Set-valued Methods'' (2017AYM8XW) and grant PdR 2016-2018-2020 ``Metodi Variazionali ed Equazioni Differenziali'' of the University of Catania.}
\author[xmu]{Shibo Liu}
\ead{liusb@xmu.edu.cn}
\author[cat]{Sunra Mosconi\corref{cor}}
\ead{mosconi@dmi.unict.it}
\cortext[cor]{Corresponding author.}

\address[xmu]{School of Mathematical Sciences, Xiamen University, Xiamen 361005, China}
\address[cat]{Dipartimento di Matematica e Informatica, Universit\`a degli Studi di Catania\\
Viale A.\ Doria 6, 95125 Catania, Italy}

\begin{abstract}
We consider a class of stationary Schr\"{o}dinger-Poisson systems with a general nonlinearity $f(u)$ and coercive sign-changing potential $V$ so that the Schr\"{o}dinger operator $-\Delta+V$ is indefinite. Previous results in this framework required $f$ to be strictly $3$-superlinear, thus missing the paramount case of the Gross-Pitaevskii-Poisson system, where $f(t)=|t|^{2}t$; in this paper we fill this gap, obtaining non-trivial solutions when $f$ is not necessarily $3$-superlinear.
\end{abstract}

\begin{keyword}

\MSC 35J60 \sep 58E05

\end{keyword}

\end{frontmatter}


\section{Introduction}

The dynamic of a Bose-Einstein condensate can be described (see \cite{MR2680421, MR2143817}) by the Gross-Pitaevskii equation
\[
i\, \partial_{t}\psi=-\Delta\psi+ V\, \psi+g\,|\psi|^{2}\psi
\]
where $\psi:\R^{3}\times [0, +\infty [ \to \mathbb{C}$ is the wave function of the condensate, $V=V(x)$ is the potential, $|\psi|^{2}$ is the particle-density, whose integral gives the total (large) number of particles $N$ and $g$ is related to the scattering length of the mutual short-range atomic interaction (resulting in positive $g$ for repulsive interaction and negative for attractive ones). The Gross-Pitaevskii equation is a particular case of the nonlinear Schr\"{o}dinger equation
\[
i\, \partial_{t}\psi=-\Delta\psi+ V\, \psi+g\, |\psi|^{p-1}\psi
\]
for $p>1$, which is usually called {\em focusing} NLS if $g<0$ and {\em defocusing} NLS if $g>0$.
If the particles are electrically charged, long-range electrostatic interaction can be effectively modelled by a potential term (see \cite{MR1890644} for a formal justification), so that  $V=V_{\rm ext} +\phi$, where $V_{\rm ext}$ is the external potential and $\phi$ is the electrostatic potential determined by the Poisson equation with charge density $k |\psi|^{2}$ (typically,  $k>0$ giving repulsive interactions).
 This gives rise to the Schr\"odinger-Poisson system
\begin{equation}
\label{SP00}
\begin{cases}
i\, \partial_{t}\psi=-\Delta\psi+ (V_{\rm ext} +\phi) \psi+g\, |\psi|^{p-1}\psi\\[8pt]
-\Delta \phi=k\, |\psi|^{2}
\end{cases}
\end{equation}
which has been object of extensive studies in the last decades. Assuming vanishing boundary conditions at infinity, the total energy
\[
E:=\int \frac{1}{2}|\nabla \psi|^{2}+ \frac{V_{\rm ext}}{2}|\psi|^{2}+\frac{k}{4}\, |\nabla\phi|^{2}+\frac{g}{p+1}|\psi|^{p+1}\, dx
\]
is conserved along the motion, as well as the total mass $N=|\psi|_2^{2}$, where $|\cdot|_q$ stands for the $L^q$-norm over $\mathbb{R}^3$. A particularly interesting case of the previous system is when long- and short-range mutual strengths compete with each other, for example when $k>0$ (repulsive electrostatic interaction) and $g<0$ (short-range binding). This is the case for a Bose-Einstein condensate of charged ions with attractive interatomic interaction, trapped in a potential well.

Standing waves for \eqref{SP00} are obtained through the ansatz $\psi(x, t)=e^{-i\omega t}u(x)$ with $u:\R^{3}\to \R$, giving
\begin{equation}
\label{e1}
\begin{cases}
-\Delta u+V\, u+\phi\, u - f(u)=0\\
-\Delta\phi= u^{2}
\end{cases}
\end{equation}
where we set $f(t)=|t|^{p-1}t$, $V=V_{\rm ext}-\omega$ and $k=1$ for simplicity of notation.
Conservation of total energy $E$ and mass $N$ gives the relation $\omega\, N=E$, so that $\omega$ is the energy per particle of the standing wave. A natural question, to which we will give a positive answer in the present paper, is wether standing waves of arbitrarily large energy per particle can occur. Notice that for large values of $\omega$, the potential $V=V_{\rm ext}-\omega$ is sign-changing and the linearisation of the first equation  turns out to be an indefinite Schr\"{o}dinger operator.

From the mathematical point of view, the energy functional $\mathcal{E}%
:H^{1}(\mathbb{R}^{3})\times\mathcal{D}^{1,2}(\mathbb{R}^{3})\rightarrow
\mathbb{R}$ given by%
\[
\mathcal{E}(u,\phi)=\frac{1}{2}\int\left(  \left\vert \nabla u\right\vert
^{2}+V\, u^{2}\right)\mathrm{d}x  -\frac{1}{4}\int\left\vert \nabla\phi\right\vert
^{2}\,\mathrm{d}x+\frac{1}{2}\int\phi \, u^{2}\,\mathrm{d}x-\int F(u)\,\mathrm{d}x
\]
where
\[
F(t)=\int_{0}^{t}f(s)\,\mathrm{d}s\text{,}
\]
is such that critical points $\left(
u,\phi\right)  $ of $\mathcal{E}$ are solutions of \eqref{e1}. However, since $\mathcal{E}$ is strongly indefinite and thus difficult to deal with, Benci {\em et al.}\,\cite{MR1659454,MR1714281} proposed the
following reduction procedure. For $u\in H^{1}(\mathbb{R}^{3})$ let $\phi
_{u}\in\mathcal{D}^{1,2}(\mathbb{R}^{3})$ be the unique solution of  $-\Delta\phi=u^{2}$ in \eqref{e1}. Then,  $u$ is a critical point of
\begin{equation}
J(u):=\frac{1}{2}\int\left(  \left\vert \nabla u\right\vert ^{2}+V\, u^{2}%
\right)\mathrm{d}x  +\frac{1}{4}\int\phi_{u}\, u^{2}\,\mathrm{d}x-\int F(u)\,\mathrm{d}x\text{,}\label{eJ}
\end{equation}
if and only if  $\left(  u,\phi_{u}\right)  $ solves \eqref{e1}; see \cite{MR1659454,MR1714281} or \cite[pp.
4929--4932]{MR2548724} for more details.

Based on this reduction method, many results on the system \eqref{e1} appeared
in the last twenty years which we will now briefly review, starting from those assuming that the associated Schr\"{o}din\-g\-er operator $-\Delta +V$ is positive definite. For  $V\equiv 1$ and $f(u)=|u|^{p-1}u$ with $p\in \left]1,5\right[$, Ruiz \cite{MR2230354} thoroughly investigated the existence of solutions for \eqref{e1}. The first paper on Schr\"{o}dinger-Poisson systems with non-constant potential seems to be Wang and Zhou \cite{MR2318269}, where $f$ is asymptotically linear. The asymptotically linear case was also studied by Sun \emph{et al.}\ \cite{MR2733219}, where as in an earlier paper \cite{MR2439518} by Mercuri, the potential is radial and vanishing at infinity. System \eqref{e1} with such potentials was also studied by Liu and Huang \cite{MR3128425} for sublinear nonlinearities.
Azzolini and Pomponio \cite{MR2422637} obtained a ground state  for $p\in\left]2,5\right[$ and positive constant potentials, while under assumption
\begin{equation}
\infty> \sup_{\R^{3}}V=\lim_{|y|\to\infty}V(y)
\label{fpww}
\end{equation}
they obtained a ground state for $p\in\left]3,5\right[$. Assuming \eqref{fpww} and some additional properties on $V$, Zhao and Zhao \cite{MR2428280} obtained ground states also for $p\in\left]2,3\right]$. Moreover, they considered periodic potentials as well, proving the existence of infinitely many
geometrically distinct solutions for $3$-superlinear nonlinearities $f$, i.e.
\begin{equation}
\label{4superl}
\lim_{|t|\to +\infty}\frac{f(x, t)t}{t^{4}}=+\infty\qquad \text{uniformly in $x\in \R^{3}$}.
\end{equation}
Results on \eqref{e1} with periodic potential can also be found in \cite{MR2769159,MR3427661}. For other types of potential we mention Chen and Tang \cite{MR2548724}, where $V$ is in some sense coercive, so that the working space can be compactly embedded into $L^2(\mathbb{R}^3)$. For $3$-superlinear, odd nonlinearities they obtained a sequence of solutions $\{u_n\}$ such that $J(u_n)\to+\infty$. Jiang and Zhou \cite{MR2802025} studied the steep potential well case $V(x)=1+\mu g(x)$, where $g(x)\ge0$ is such that $g^{-1}(0)$ is bounded and has nonempty
interior. Then, for pure power nonlinearities, nontrivial solutions are obtained for sufficiently large $\mu$. Asymptotic behavior of the solutions as $\mu\to +\infty$ is also investigated.

Cerami and Vaira \cite{MR2557904} obtained a ground state  for the generalized Schr\"{o}dinger-Poisson system
\begin{equation}
\left\{
\begin{array}{ll}
-\Delta u+V\, u+K\, \phi\, u=Q\, |u|^{p-1}u\text{,}\\
-\Delta \phi=K\, u^2\text{,}
\end{array}
\right.
\label{cvjde}
\end{equation}
where $V\equiv 1$, $p\in\left]3,5\right[$, $K$ and $Q$ are nonnegative functions on $\mathbb{R}^3$ satisfying
suitable assumptions. It is also possible to obtain solutions when no ground state exists, see {\em e.g.}\ \cite{MR3551056} for $p\in \ ]3, 5[$ and \cite{CM1} in the critical case $p=5$ with $Q\equiv 1$. Sun \emph{et al.} \cite{MR3411683}  found at least $k$ positive solutions for \eqref{cvjde}  for sufficiently large $\lambda$ assuming $Q$ has $k$ strict positive maxima.

Concerning sign-changing solutions of \eqref{e1}, the first result seems to be obtained by Ianni \cite{MR3114313} via a dynamical approach for the case $V\equiv 1$ and $f(u)=|u|^{p-1}u$ with $p\in\left[3, 5\right[$; see also Kim and Seok \cite{MR2989645} for a similar result. Notice that the relations
\[
J(u)=J(u^+)+J(u^-),\qquad\langle J'(u),u^\pm\rangle=\langle J'(u^\pm),u^\pm\rangle
\]
where $u^\pm=\max\{\pm u,0\}$ {\em fail} to be true  for $J$, thus posing additional difficulties when trying to construct sign-changing solutions if $V$ is not a constant. In this direction, the first breakthrough seems to be made by Wang and Zhou \cite{MR3311919} in the case $f(u)=|u|^{p-1}u$ with $p\in\left]3,5\right[$. Motivated by Bartsch and Weth \cite{MR2136244}, they sought minimizer for $J$ over a constraint $\mathcal{M}$ and showed that the minimizer is a sign-changing solution via degree theory. Using the method of invariant sets of descending flow, Liu \emph{et al.} \cite{MR3500305} obtained sign-changing solutions for $3$-superlinear nonlinearities. In these papers a compactness condition related to $V$ was assumed, which fails if $V$ obeys {\em e.g.}\ \eqref{fpww} (finite potential well). A least energy sign-changing solution assuming \eqref{fpww} and \eqref{4superl} was obtained by Alves \emph{et al.} \cite{MR3619755}.

We emphasize that in the aforementioned papers the Schr\"{o}dinger
operator $-\Delta+V$ is always assumed to be positive definite (as when $\inf
_{\mathbb{R}^{3}}V>0$). With this assumption $u\equiv 0$ is a
local minimizer of $J$, leading to a mountain pass geometry if $f$ is $3$-superlinear (actually, superquadratic will often suffices, but require more intricate arguments).
However, if we seek for standing waves with
large  $\omega$, then $V=V_{\rm ext}-\omega$ will be negative somewhere, disrupting the mountain pass
geometry.
For stationary NLS equations
\begin{equation}
-\Delta u+V\, u=f(u)\text{\qquad in }
\mathbb{R}^{N}\text{,} \label{xm2}%
\end{equation}
with indefinite Schr\"{o}dinger operator $-\Delta+V$, one usually applies the linking theorem to get solutions, see {\em e.g.}\ \cite{MR1751952,MR2957647}.
For system \eqref{e1}, however, because the term
involving $\phi_{u}$ in the functional $J$ is nonnegative, $J$ may be positive
somewhere on the negative space of  $-\Delta+V$. It thus seems hard to verify the linking geometry and get critical points of
$J$ via the linking theorem; see \cite[p. 47]{MR3303004} for further discussion on this issue. This is probably one of the
reasons why there are very few existence results for
\eqref{e1} if the Schr\"{o}dinger operator $-\Delta+V$ is indefinite.

In \cite[Theorem 1.1]{MR3045632}, Zhao \emph{et al.} studied the following system%
\begin{equation}
\left\{
\begin{array}
[c]{ll}%
-\Delta u+\lambda\, V\, u+K\, \phi\, u=\left\vert u\right\vert ^{p-1}u\text{,} \\
-\Delta\phi=K\, u^{2}\text{,}
\end{array}
\right.  \label{xm3}%
\end{equation}
Here $p\in\left]  3, 5\right[  $, $K$ is a nonnegative function while $V$ may be negative somewhere. When $\lambda>0$ is sufficiently large and $|
K| _{2}$ (or $| K|_{\infty}$) is sufficiently small, they got a
nontrivial solution for \eqref{xm3} through the linking theorem (see also Ye and Tang
\cite[Theorem 1.1]{MR3336325} for a generalization involving $3$-superlinear reactions).
Note that if $K\equiv0$, the system \eqref{xm3} reduces to the single equation \eqref{xm2}, whose
functional exhibits a linking structure near $0$. It is thus natural to expect that the
functional for \eqref{xm3} also possesses the linking structure if $K$ is small
in suitable sense, as required in \cite{MR3336325,MR3045632}.

The first result without any smallness assumption on the factor of $\phi\, u$ is due to Chen and Liu
\cite{MR3303004}, where a nontrivial solution for \eqref{e1} was obtained when $V$ is coercive in the sense
\begin{itemize}
\item[$\left(  V_{0}\right)  $]$V\in C(\mathbb{R}^{3})$ is bounded from below
and $\left\vert \left\{  V\leq k\right\}  \right\vert <\infty$ for all
$k\in\mathbb{R}$,
\end{itemize}
and $f$ is assumed to be $3$-superlinear and subcritical.
The key observation of \cite{MR3303004} is that although $J$
may not posses the linking geometry, it nevertheless has a local linking at $0$ as soon as $f$ is superlinear near $0$. Critical points for $J$ can thus be obtained via the local linking
theory \cite{MR1312028,MR2135818}.

Notice that condition $(V_{0})$ ensures that for some $m>2$, $\tilde V:=V+m>1$ and that the embedding
\[
X:=\left\{u\in H^{1}(\R^{3}): \int \tilde V\, u^{2}\, \mathrm{d}x<+\infty\right\}\hookrightarrow L^{2}(\R^{3})
\]
is compact, where the norm in $X$ is given by
\[
\|u\|=\left(\int\left[|\nabla u|^2+\tilde{V}u^2\right]\, \mathrm{d}x\right)^{1/2},\qquad u\in X.
\]
In general, the natural working space fails to compactly embed into $L^{2}(\R^{3})$ and in
\cite{MR3656292}, Liu and Wu studied the case $V\in C(\mathbb{R}^{3})$ being
bounded. Assuming $f(x,t)=a(x)\left\vert t\right\vert ^{p-1}t$
where $a> 0$ vanishing at infinity and $p\in\left]
3,5\right[  $, they obtained a nontrivial solution for \eqref{xm3} via local linking and
Morse theory \cite{MR1196690,MR982267}.

We underline that in the previous papers
\cite{MR3303004,MR3656292,MR3336325,MR3045632}, the nonlinearity is
$3$-superlinear, i.e. \eqref{4superl} holds and, to the best of our knowledge, currently there is no result for the
Schr\"{o}dinger-Poisson system \eqref{e1} with indefinite potential and
nonlinearity not $3$-superlinear. The relevance of this latter framework is clear from the introductory discussion on Bose-Einstein condensates,  since the nonlinearity $f(t)=|t|^{2}t$ corresponding to the Gross-Pitaevskii equation is exactly $3$-linear.

Our first result treats subquadratic nonlinearities. By $\sigma(-\Delta+V)$ we mean the spectrum of $-\Delta+V$, which is understood as the natural self-adjoint operator on $X$ corresponding to the bilinear form given by \eqref{defQ} below.

\begin{theorem}
\label{thc}Suppose that $\left(  V_{0}\right)  $ hold and $\inf\sigma
(-\Delta+V)\leq0$. If there
 exist $C,v>0$, $p,q\in
 \left]  1,2\right[  $ such
that%
\begin{equation}\label{xmxm}
\left\vert f(t)\right\vert \leq C(  \left\vert t\right\vert
^{p}+\left\vert t\right\vert ^{q})
\end{equation}
and $F(t)\geq c\left\vert t\right\vert ^{p+1}$ for all $t\in\mathbb{R}$, then
there are at least two nontrivial solutions to \eqref{e1}.
\end{theorem}

Under the stated assumptions, the functional
$J$ is coercive, hence $\left(  PS\right)$ sequences
are automatically bounded and precompact by $(V_{0})$. Since $J$ has a local linking at $0$, Theorem \ref{thc} follows from  the three critical points theorem of Liu \cite[Theorem 2.2]{MR1110119}.

Our next and main result deals with the superquadratic case, which includes $f(u)=\left\vert u\right\vert
^{p-1}u$ with $p\in\left]
2,5\right[  $. In addition to the basic assumptions $\left(  V_{0}\right)  $, we will need the following
\begin{itemize}
\item[$\left(  V_{1}\right)  $] $V\in C^{1}(\mathbb{R}^{3})$ and there exists
$R>0$ such that
\[
W(x):=2V(x)+\nabla V(x)\cdot x\geq0\text{\qquad for }\left\vert x\right\vert
\geq R\text{.}%
\]

\item[$\left(  V_{2}\right)  $] There exists $\kappa>0$ such that%
\[
\left\vert \nabla V(x)\cdot x\right\vert \leq\kappa\left(  V(x)+m\right)
\text{\qquad for all }x\in\mathbb{R}^{3}\text{.}%
\]

\item[$\left(  f_{1}\right)  $] $f\in C(\mathbb{R})$ and there exists $C>0$,
$p\in\left(  1,5\right)  $ such that%
\[
\left\vert f(t)\right\vert \leq C(  \left\vert t\right\vert +\left\vert
t\right\vert ^{p})  \text{\qquad for all }t\in\mathbb{R}\text{,}%
\]

\item[$\left(  f_{2}\right)  $] There exists $\mu>3$ such that%
\[
f(t)t\geq\mu F(t)>0\text{\qquad for all }t\in\mathbb{R}\backslash\left\{
0\right\}  \,\text{.}%
\]
\end{itemize}
It is easy to check that any coercive, radially increasing potential
with at most polynomial growth satisfies our assumptions, an explicit
example being $V(x)=|x|^2-1$.
For a more detailed discussion on assumptions $(V_{1})$ and $(V_{2})$, we refer to the beginning of Section 3.
Moreover, we let $X_+$, $X_-$ and $X_0$ denote  the positive, negative and null eigenspaces of
the Schr\"{o}dinger operator, respectively.

\begin{theorem}
\label{t5.12}Assume $\left(  V_{0}\right)  $--$\left(  V_{2}\right)  $,
$\left(  f_{1}\right)  $--$\left(  f_{2}\right)  $ hold. If either

\begin{enumerate}
\item $\dim X_{-}>0$, $\dim X_{0}=0$

\item $\dim X_{0}>0$ and $F(t)\geq c\left\vert t\right\vert ^{\nu}$ for some
$\nu<4$,
\end{enumerate}
then problem \eqref{e1} has at least a nontrivial solution.
\end{theorem}

Let us discuss some features of Theorem \ref{t5.12} in the model case $f(u)=\left\vert
u\right\vert ^{p-1}u$. The main difficulty in studying the Schr\"{o}dinger-Poisson system in the range $p\in \ ]2, 3]$ is that
it is not known (even in the easiest setting $V\equiv 1$) wether $\left(  PS\right)$ sequences are bounded or not. This issue disappears in the $3$-superlinear setting $p>3$, allowing an easier application of variational methods (this is also why most of the previous papers on the subject assume \eqref{4superl}). To overcome this
difficulty there are typically two approaches:
\begin{itemize}
\item
seek for a minimum on a suitable manifold $\mathcal{M}$. Coercivity is still an issue for $p\in \ ]2, 3]$ and  the standard Nehari manifold won't help, so one usually works on the Pohozaev manifold (or variants of it);
\item
employ Struwe's monotonicity trick, i.e., define monotonic perturbations of $J$ and find a solution for almost every perturbation. Boundedness of the resulting sequence is proved via the Pohozaev identity.
\end{itemize}
Both approaches can be successful when the Schr\"{o}dinger operator $-\Delta+V$ is positive definite, but run into serious issues when it is not, for reasons which we will briefly outline.

For NLS \eqref{xm2} with indefinite potential, the Nehari manifold $\mathcal{N}$ can be modified as in Szulkin and Weth \cite{MR2557725} to produce ground states. This is possible thanks to a linearity feature of the Nehari manifold, namely that $\mathcal{N}$ is the set of critical points of $J$ along {\em lines} through $0$. On the contrary, the natural curves defining  the Pohozaev manifold (which is the one apparently needed to get coercivity)  are highly nonlinear and may have nothing to do with the orthogonal decomposition of the space dictated by the linear operator $-\Delta+V$.

On the other hand, Struwe's monotonicity trick is usually successful when a uniform mountain pass geometry or more general linking geometry holds for the family of monotonic perturbations of $J$, see \cite{MR1718530,MR1867339} respectively. This cannot hold for the indefinite Schr\"{o}dinger-Poisson systems we are considering because, as pointed out in \cite[p. 47]{MR3303004},  our functional $J$ only has a {\em local} linking at the origin. Currently, it seems unclear how to implement the monotonicity trick in a local linking geometry, due to the lack of an explicit minimax description of the critical values in this setting.

To get around these difficulties, motivated by \cite{MR1430506}, we consider an
augmented functional $\tilde{J}:\mathbb{R}\times X\rightarrow\mathbb{R}$, see
\eqref{argeq}. It turns out that $\tilde{J}$ easily satisfies the $\left(  PS\right)  $
condition (Theorem \ref{PS}), and if $\left(  \bar{s},\bar{u}\right)  $ is
critical point of $\tilde{J}$, then $\bar{u}$ is critical point of $J$ (Lemma
\ref{crt}). Moreover, $\tilde{J}$ has a local linking at $(0, 0)$ and we can compute the homology of $\tilde{J}$ at infinity, so that eventually we
will apply Morse theory to get a critical point of $\tilde{J}$ and thus of $J$.

The paper is organized as follows. In Section 2 we recall the functional analytic tools we'll need and prove Theorem
\ref{thc}. In Section 3, we deal with the superquadratic case and present the
proof of Theorem \ref{t5.12}.

\section{The coercive case}

Let us discuss some first consequences of assumption $(V_{0})$. From the lower boundedness we will henceforth fix $m>2$ such that
\begin{equation}
\label{em}
\tilde V(x):=V(x)+m>\frac{m}{2}>1,\qquad \text{for all $x\in \R^{3}$}.
\end{equation}
As already mentioned, by \cite{MR1349229} we see that the Hilbert space
\[
X:=\left\{u\in H^{1}(\R^{3}): \int \tilde V\, u^{2}\, \mathrm{d}x<+\infty\right\},\qquad (u, v)_{X}=\int\left[\nabla u\cdot \nabla v+\tilde V\, u\, v\right]\, \mathrm{d}x
\]
 compactly embeds in $L^{2}(\R^{3})$.  Notice that since $X$ also embeds into $L^{6}(\R^{3})$, by interpolation, $X\hookrightarrow L^{r}(\R^{3})$ compactly for all $r\in [2, 6[$. From the compactness of $X\hookrightarrow L^{2}(\R^{3})$, we deduce that the bilinear form
\begin{equation}
\label{defQ}
Q(u, v)=\frac12\int \left(\nabla u\cdot \nabla v+V\, u\, v\right)\, \mathrm{d}x, \qquad u, v\in X,
\end{equation}
is essentially selfadjoint (by Kato's criterion), semibounded from below on $X\subseteq L^{2}(\R^{3})$ and the spectrum of the corresponding Schr\"{o}dinger operator $\sigma(-\Delta+V)$ is discrete (with finite multiplicity) and bounded from below.
In the following, we will denote by $X_{+}$, $X_{-}$ and $X_{0}$ respectively the positive, negative and null eigenspaces of the Schr\"{o}dinger operator, and by $u\mapsto u_{\pm}$ and $u\mapsto u_{0}$ the corresponding orthogonal projections. Accordingly, there exists $\lambda_{\pm}>0$ such that
\begin{equation}
\label{pm}
\pm Q(u, u)\ge \lambda_{\pm }\|u_{\pm}\|^{2}\qquad \text{for $u\in X_{\pm}\oplus X_{0}$, respectively}.
\end{equation}

As already pointed out, solving \eqref{e1} is equivalent to finding critical points of the $C^{1}$ functional $J:X\rightarrow\mathbb{R}$,
\[
J(u)=\frac{1}{2}\int \left[\left\vert \nabla u\right\vert ^{2}
+V\,u^{2}\right]\, \mathrm{d}x  +\frac{1}{4}\int\phi_{u}\, u^{2}\, \mathrm{d}x-\int F( u)\, \mathrm{d}x,
\]
where $\phi_{u}$ is the unique solution of $-\Delta\phi=u^{2}$ in
$\mathcal{D}^{1,2}(\mathbb{R}^{3})$. Recall  that
\begin{equation}
\label{ct}
0\le \int\phi_{u}\, u^{2}\, \mathrm{d}x\le C\|u\|^{4},
\end{equation}
see {\em e.g.} \cite{MR2548724}. We will also need the following estimate, whose proof is similar to \cite[Eqn (19)]{MR2230354}, therefore is omitted.

\begin{lemma}
For any $u\in H^1(\mathbb{R}^3)$ we have
\begin{equation}
\label{RI}
\int\left\vert u\right\vert ^{3}\, \mathrm{d}x\le\frac{1}{2}\int\left[\left\vert \nabla
u\right\vert ^{2}+\phi_{u}\, u^{2}\right]\, \mathrm{d}x\text{.}%
\end{equation}
\end{lemma}

Given a Hilbert space $X$, we say that a functional $J\in C^{1}(X)$ has a local linking at $0$ if $X=X^{-}\oplus X^{+}$ for some closed proper subspaces $X^{\pm}$ and for some $\rho>0$ there holds
\[
\begin{cases}
J>0&\text{in $B_{\rho}\cap (X^{+}\setminus \{0\})$},\\
J\le0&\text{in $B_{R}\cap X^{-}$}.
\end{cases}
\]
This implies that $u=0$ is a trivial critical point of $J$. The following three critical point theorem can be found in Liu \cite[Theorem 2.2]{MR1110119}, which is a special case of \cite[Theorem 2.1]{MR1828101}.
\begin{theorem}\label{Perera}
Let $J\in C^{1}(X)$ satisfy the $(PS)$-condition, have a local linking at $0$ with $\dim X^-<\infty$, and be bounded from below. Then $J$ has at least two nontrivial critical points.
\end{theorem}

Now we can start our investigation for the functional $J$.

\begin{proposition}
Suppose that $(V_{0})$ holds and that there exist $C\ge 0$, $p,q\in [1, 2[$ such that
\begin{equation}
|F(t)|\le C(|t|^{p+1}+|t|^{q+1})\label{hyp2}
\end{equation}
for all $t\in \R$, then $J$ is coercive on $X$.
\end{proposition}

\begin{proof}
Let us choose $m>0$ as in \eqref{em} and set $\Lambda:X\rightarrow\mathbb{R}$,
\[
\Lambda(u)=-\frac{m}{4}\int u^{2}\, \mathrm{d}x+\frac{1}{2}\int\left\vert u\right\vert
^{3}\, \mathrm{d}x+\frac{1}{4}\int V\, u^{2}\, \mathrm{d}x-\int |F(u)|\, \mathrm{d}x\text{.}%
\]
For $u\in X$, using \eqref{RI} we have
\begin{align*}
J(u)  &  =\frac{1}{2}\int\left[ \left\vert \nabla u\right\vert
^{2}+V\, u^{2}\right]\, \mathrm{d}x  +\frac{1}{4}\int\phi_{u}\, u^{2}\, \mathrm{d}x-\int F( u)\, \mathrm{d}x\\
&  =\frac{1}{4}\int\left[\left\vert \nabla u\right\vert ^{2}+V\, u^{2}\right]\, \mathrm{d}x+\frac{1}{4}
\int\left[\left\vert \nabla u\right\vert ^{2}+\phi_{u}\, u^{2}\right]\, \mathrm{d}x  +\frac{1}
{4}\int V\, u^{2}\, \mathrm{d}x-\int F(u)\, \mathrm{d}x\\
&  \ge\frac{1}{4}\left\Vert u\right\Vert ^{2}-\frac{m}{4}\int u^{2}\, \mathrm{d}x+\frac
{1}{2}\int\left\vert u\right\vert ^{3}\, \mathrm{d}x+\frac{1}{4}\int V\, u^{2}\, \mathrm{d}x-\int |F(u)|\, \mathrm{d}x\\
&  =\frac{1}{4}\left\Vert u\right\Vert ^{2}+\Lambda(u)\text{.}%
\end{align*}
Therefore, it suffices to show that the functional $\Lambda$ is bounded from below.

For any  $M>m$, since $V(x)\ge-m$ for all $x\in\R$, we have
\[
\int V\, u^{2}\, \mathrm{d}x=\int_{\{V>M\}}V\, u^{2}\, \mathrm{d}x+\int_{\{V\le M\}}V\, u^{2}\, \mathrm{d}x\ge M\int
_{\{V>M\}}u^{2}\, \mathrm{d}x-m\int_{\{V\le M\}}u^{2}\, \mathrm{d}x\text{,}%
\]
so that
\begin{equation}
\Lambda(u)     \ge \frac{M}{4}\int_{\{V>M\}}u^{2}\, \mathrm{d}x-\frac{m}{4}\int_{\{V\le M\}}u^{2}\, \mathrm{d}x-\frac{m}{4}\int u^{2}\, \mathrm{d}x+\frac{1}
{2}\int\left\vert u\right\vert ^{3}\, \mathrm{d}x-\int |F(u)|\, \mathrm{d}x.\label{Lamd}
\end{equation}
Accordingly,  we split all the remaining integrals on the two sets $\{V>M\}$ and $\{V\le M\}$, proving boundedness of the corresponding quantities separately.

On $\{V\le M\}$, which has finite measure by assumption $(V_0)$, H\"older inequality gives
\[
\int_{\{V\le M\}}u^{2}\, \mathrm{d}x\le C_{2}\left(  \int_{\{V\le M\}}\left\vert
u\right\vert ^{3}\, \mathrm{d}x\right)  ^{\frac 2 3}.
\]
Similarly, using \eqref{hyp2} as well,
\[
\int_{\{V\le M\}} |F(u)|\, \mathrm{d}x\le C_{p}\left(  \int_{\{V\le M\}}\left\vert
u\right\vert ^{3}\, \mathrm{d}x\right)^{\frac{p+1}{3}}+C_{q}\left(  \int_{\{V\le M\}}\left\vert
u\right\vert ^{3}\, \mathrm{d}x\right)^{\frac{q+1}{3}}
\]
for some constants $C_{r}$ depending on $M$, $V$ and $r\in [1, 2[$.  Hence
\begin{align*}
\Lambda_{M}^{-}(u)&:=-\frac{m}{2}\int_{\{V\le M\}}u^{2}\, \mathrm{d}x + \frac{1}{2}\int_{\{V\le M\}}\left\vert u\right\vert ^{3}\, \mathrm{d}x-\int_{\{V\le M\}} |F(u)|\, \mathrm{d}x\\
&\ \ge \frac{1}{2}\int_{\{V\le M\}}\left\vert u\right\vert ^{3}\, \mathrm{d}x
-\frac{m\, C_{2}}{2}\left(  \int_{\{V\le M\}}\left\vert u\right\vert ^{3}\, \mathrm{d}x\right)  ^{\frac 2 3}\\
&\qquad\qquad-C_{p}\left(  \int_{\{V\le M\}}\left\vert u\right\vert ^{3}\, \mathrm{d}x\right)^{\frac{p+1}{3}}-C_{q}\left(  \int_{\{V\le M\}}\left\vert u\right\vert ^{3}\, \mathrm{d}x\right)^{\frac{q+1}{3}}
\end{align*}
and since  $q,p\in\ [ 1,2[$, for any choice of $M, m$ the right hand side is clearly bounded from below.

Consider now
\begin{align}
\Lambda_{M}^{+}(u)&:=\frac{M-m}{4}\int_{\{V>M\}}u^{2}\, \mathrm{d}x+ \frac{1}{2}\int_{\{V>M\}}\left\vert u\right\vert ^{3}\, \mathrm{d}x-\int_{\{V>M\}}|F(u)|\, \mathrm{d}x\nonumber\\
&\ge \frac{M-m}{4}\int_{\{V>M\}}u^{2}\, \mathrm{d}x+ \frac{1}{2}\int_{\{V>M\}}\left\vert u\right\vert ^{3}\, \mathrm{d}x\nonumber\\
&\qquad\qquad-C\int_{\{V>M\}}|u|^{p+1}\, \mathrm{d}x-C\int_{\{V>M\}}|u|^{q+1}\, \mathrm{d}x .\label{e2}
\end{align}
For $r\in\{ p+1, q+1\}$, by the interpolation inequality we have
\begin{equation}
\label{e3}
\int_{\{V>M\}}\left\vert u\right\vert ^{r}\, \mathrm{d}x\le\left(  \int_{\{V>M\}}u^{2}\, \mathrm{d}x\right)
^{\frac{r\theta_{r}}{2}}\left(  \int_{\{V>M\}}\left\vert u\right\vert ^{3}\, \mathrm{d}x\right)
^{\frac{r\left(  1-\theta_{r}\right)  }{3}}
\end{equation}
for $\theta_{r}\in\ [ 0,1[  $ satisfying
\[
\frac{r\, \theta_{r}}{2}+\frac{r\left(  1-\theta_{r}\right)  }{3}=1\text{.}
\]
We can suppose that $u\neq 0$ on $\left\{  V>M\right\}  $ (otherwise $\Lambda_{M}^{+}(u)=0$) and set
\[
R_{u}=\left(  \int_{\{V>M\}}u^{2}\, \mathrm{d}x\right)  ^{-1}\int_{\{V>M\}}\left\vert u\right\vert
^{3}\, \mathrm{d}x\text{.}
\]
Then applying \eqref{e3} for $r=q+1, p+1$ to \eqref{e2} we have
\begin{align*}
\Lambda_{M}^{+}(u) &  \ge\frac{M-m}{4}\int_{\{V>M\}}u^{2}\, \mathrm{d}x+\frac{1}{2}\int_{\{V>M\}}\left\vert
u\right\vert ^{3}\, \mathrm{d}x \\
&\qquad\qquad-C\left(  \int_{\{V>M\}}u^{2}\, \mathrm{d}x\right)  ^{\frac{(p+1)\theta_{p+1}
}{2}}\left(  \int_{\{V>M\}}\left\vert u\right\vert ^{3}\, \mathrm{d}x\right)  ^{\frac{(p+1)\left(
1-\theta_{p+1}\right)  }{3}}\\
&\qquad\qquad -C\left(  \int_{\{V>M\}}u^{2}\, \mathrm{d}x\right)  ^{\frac{(q+1)\theta_{q+1}
}{2}}\left(  \int_{\{V>M\}}\left\vert u\right\vert ^{3}\, \mathrm{d}x\right)  ^{\frac{(q+1)\left(
1-\theta_{q+1}\right)  }{3}}\\
&  =   \left(  \frac{M-m}{4}+\frac{R_{u}}{2}
-C\, R_{u}^{\frac{(p+1)\left(  1-\theta_{p+1}\right)  }{3}}-C\, R_{u}^{\frac{(q+1)\left(  1-\theta_{q+1}\right)  }{3}}\right)\int_{\{V>M\}}u^{2}\, \mathrm{d}x.
\end{align*}
Because  $\frac{r\left(  1-\theta_{r}\right)  }{3}<1$ for $r=p+1, q+1$, there exists $M>m$ such that
\begin{equation}
\label{e4}
\frac{M-m}{4}+\frac{R}{2}-C\, R^{\frac{(p+1)\left(  1-\theta_{p+1}\right)  }{3}}-C\, R^{\frac{(q+1)\left(  1-\theta_{q+1}\right)  }{3}}>0\text{,\qquad for all }R>0.
\end{equation}
With this choice of $M$ at the very beginning, the above argument shows that $\Lambda_{M}^{+}(u)\ge 0$. Since $ \Lambda_{M}^{-}(u)+\Lambda_{M}^{+}(u)$ is exactly the right hand side of \eqref{Lamd}, we deduce that $\Lambda$ is bounded from below and the proof is concluded.
\end{proof}

\begin{remark}
\quad
\begin{itemize}
\item
Regarding the potential, coercivity  holds under slightly weaker assumptions, namely that the measure of $\{V\le M\}$ is finite for a suitable large $M$ prescibed by the validity of \eqref{e4}. However, without assuming the full $(V_0)$ in the previous proposition, compactness starts becoming the main issue to prove existence of a solution.
\item
The case $1\le p<q=2$ can also be treated but is border-line: consider in \eqref{e1} the nonlinearity $f(t)=\lambda |t| t$ for $\lambda>0$. The previous proof still works for $\lambda\le \lambda_{0}$ being $\lambda_{0}$ a small positive number that can be explicitly computed, but fails for $\lambda>\lambda_{0}$. The arguments in \cite[Theorem 4.1]{MR2230354} show that for $\lambda>\lambda_{0}$ there are actually no solutions.
\end{itemize}
\end{remark}

Now we can give the proof of Theorem \ref{thc}.
\begin{proof}[Proof of Theorem \ref{thc}]
The arguments of \cite[p.4933]{MR2548724} and the coercivity of  $J$ imply that the $(PS)$ condition holds.  Let $X_{-}$, $X_{+}$ and $X_{0}$ be the negative, positive and zero eigenspaces of the bilinear form $Q$ defined in \eqref{defQ}, with $u_{-}$, $u_{+}$ and $u_{0}$ being the respective orthogonal projections of $u$.  We claim that $J$ has a local linking at $0$ with respect to the decomposition $(X_{-}\oplus X_{0})\oplus X_{+}$.
By the compactness of $X\hookrightarrow L^{2}(\R^{3})$, both $X_{-}$ and $X_{0}$ are finite dimensional and $(X_{-}\oplus X_{0})$ has positive dimension because $\inf \sigma(-\Delta +V)\le 0$. By the embedding $X\hookrightarrow L^{r}(\R^{3})$ for $r\in\{p+1, q+1\}$, there holds
\[
\left|\int F(u)\, \mathrm{d}x\right|\le C\left(\|u\|^{p+1}+\|u\|^{q+1}\right),
\]
so that
\[
J(u)=Q(u, u)+G(u)
\]
where, recalling that $p, q>1$ and \eqref{ct}
\[
G(u):=\int \left[\frac14\phi_{u}\, u^{2}-F(u)\right]\, \mathrm{d}x=o(\|u\|^{2})\text{,\qquad as }\|u\|\to0\text{.}
\]
This immediately forces $J> 0$ on $B_{R}\cap X_{+}\setminus \{0\}$ for suitably small $R>0$.
For $u$ in the finite dimensional space $X_{-}\oplus X_{0}$ (where all norms are equivalent), it holds
\[
\int F(u)\, \mathrm{d}x\ge c\int |u|^{p+1}\, \mathrm{d}x\ge \tilde c\|u\|^{p+1},
\]
for some $\tilde c>0$. Therefore, by \eqref{ct} and \eqref{pm} we deduce
\[
J(u)\le -\lambda_{-}\|u_{-}\|^{2}+C\, \|u\|^{4}-\tilde c\, \|u\|^{p+1}\qquad \text{for $u\in X_{-}\oplus X_{0}$},
\]
Since $p+1<4$, this implies that $J<0$ in $ B_{R}\cap \left(X_{-}\oplus X_{0}\right)\setminus \{0\}$ for an even smaller $R>0$.
The conclusion now follows from Theorem \ref{Perera}.
\end{proof}

\begin{remark}
The condition $F(t)\ge c |t|^{p+1}$ is only used to deal with the case $\dim X_{0}>0$. If  $\dim X_0=0$ it is not needed and the multiplicity result above holds under the s\^ole assumption \eqref{xmxm} with $p,q\in]1,2[$.
\end{remark}

\section{The superquadratic case}

Let us recall the assumptions we will use in this section to prove Theorem \ref{t5.12}:
\begin{itemize}
\item[$\left(  V_{0}\right)  $] $V\in C(\mathbb{R}^{3})$ is bounded from below
and $\left\vert \left\{  V\leq k\right\}  \right\vert <\infty$ for all
$k\in\mathbb{R}$,
\item[$\left(  V_{1}\right)  $] $V\in C^{1}(\mathbb{R}^{3})$ and there exists
$R>0$ such that
\[
W(x):=2V(x)+\nabla V(x)\cdot x\geq0\text{\qquad for }\left\vert x\right\vert
\geq R\text{.}%
\]
\item[$\left(  V_{2}\right)  $] There exists $\kappa>0$ such that%
\[
\left\vert \nabla V(x)\cdot x\right\vert \leq\kappa\left(  V(x)+m\right)
\text{\qquad for all }x\in\mathbb{R}^{3}\text{.}%
\]
\item[$\left(  f_{1}\right)  $] $f\in C(\mathbb{R})$ and there exists $C>0$,
$p\in\left(  1,5\right)  $ such that%
\[
\left\vert f(t)\right\vert \leq C(  \left\vert t\right\vert +\left\vert
t\right\vert ^{p})  \text{\qquad for all }t\in\mathbb{R}\text{,}%
\]
\item[$\left(  f_{2}\right)  $] There exists $\mu>3$ such that%
\[
f(t)t\geq\mu F(t)>0\text{\qquad for all }t\in\mathbb{R}\backslash\left\{
0\right\}  \,\text{.}%
\]
\end{itemize}

We first briefly discuss the meaning of the previous hypotheses, as well as some of their consequences.
\begin{itemize}
\item
Assumption $(V_{1})$ can be seen as a lack oscillation condition at infinity. For coercive radial potentials $V(x)=v(|x|)$ it can be rephrased requiring that $r\mapsto v(r)\, r^{2}$ is non-decreasing. An example of coercive potential failing to satisfy $(V_{1})$ is $V(x)=|x|^{2}+|x|\sin |x|^{2}$.
\item
Condition $(V_{2})$ rules out exponentially growing potentials for which the implication $u\in X\Rightarrow u(\lambda\, \cdot)\in X$ may fail for $\lambda\neq 1$. For example, if $V(x)=e^{|x|}$ and  $u=e^{-|x|}/(1+|x|^{4})$, then certainly $u\in X$  but $u(\lambda \, \cdot)$ fails to be in $X$ for any $\lambda \in \ ]0, 1[$. A quantitative version of this is given in Lemma \ref{l1}.
\item
Notice that condition $(f_{1})$ implies that $|F(t)|\le C\, (|t|^{2}+|t|^{p+1})$, $p<5$, so that $J$ is well defined on $X$. We avoid the critical case $p=5$, which would require a separate concentration-compactness analysis.
\item
Hypothesis $(f_{2})$ is a $3$-superlinear condition of Ambrosetti-Rabinowitz type. By standard arguments, it implies that $t\mapsto F(t)/|t|^{\mu-1}t$ is non-decreasing and therefore,
\begin{equation}
\label{mu0}
F(t)\le C\, |t|^{\mu} \quad \text{for $|t|\le 1$}, \qquad F(t)\ge C\, |t|^{\mu}\quad \text{for $|t|\ge 1$}.
\end{equation}
\end{itemize}

\begin{lemma}
\label{l1}
If $(V_{2})$ holds, then for any $t>0$, $x\in \R^{3}$
\begin{equation}
\label{el1}
\tilde{V}(tx)\le \max\{t^{\kappa}, t^{-\kappa}\}\, \tilde{V}(x).
\end{equation}
\end{lemma}

\begin{proof}
For $t\ge1$ we have
\begin{align*}
\log\frac{\tilde{V}(tx)}{\tilde{V}(x)}  &  =\log\tilde{V}(tx)-\log\tilde
{V}(x)=\int_{1}^{t}\frac{\mathrm{d}}{\mathrm{d}s}
\log\tilde{V}(sx)\, \mathrm{d}s\\
&  =\int_{1}^{t}\frac{\nabla\tilde{V}(sx)\cdot\left(  sx\right)  }{\tilde
{V}(sx)}\, \frac{1}{s}\, \mathrm{d}s\\
&  \le\int_{1}^{t}\frac{\vert \nabla\tilde{V}(sx)\cdot(  sx)
\vert }{\tilde{V}(sx)}\, \frac{1}{s}\, \mathrm{d}s\le\int_{1}^{t}\frac{\kappa}%
{s}\, \mathrm{d}s=\log t^{\kappa}\text{.}%
\end{align*}
Therefore $\tilde{V}(tx)\le t^{\kappa}\, \tilde{V}(x)$. The argument for the case $0<t<1$ is similar.
\end{proof}

For any $t>0$ and $u\in X$ define
\begin{equation}
\label{ut}
u_{t}(x)=t^{2}\, u(tx),
\end{equation}
and define on the Hilbert space $\R\times X$ (with natural norm $\|(s, u)\|^{2}=s^{2}+\|u\|^{2}$) the augmented functional
\begin{equation}
\label{argeq}
\tilde{J}(s, u):=\frac{s^{2}}{2}+J(u_{e^{s}}).
\end{equation}
\begin{remark}\label{rkut}
  Obviously, for $s,t>0$, from \eqref{ut} we have $(u_t)_s=u_{ts}=(u_s)_t$.
\end{remark}
\begin{proposition}\label{proptj}
Assume $(V_{2})$ and $(f_{1})$. Then the functional $\tilde J$ is well defined on $\R\times X$, of class $C^{1}$ and
\begin{equation}
\label{tj}
\tilde J(s, u)=\frac{s^{2}}{2}+\frac{e^{3s}}{2}\int \left[|\nabla u|^{2}+\frac{1}{2}\phi_{u}\, u^{2}\right]\, \mathrm{d}x+\frac{e^{s}}{2}\int V(xe^{-s})\, u^{2}\, \mathrm{d}x-e^{-3s}\int F(e^{2s}u)\, \mathrm{d}x,
\end{equation}
\begin{align}
&\langle \partial_{u}\tilde J(s, u),  \varphi\rangle=e^{3s}\int \left[\nabla u\, \nabla \varphi+\phi_{u}\, u\, \varphi\right]\, \mathrm{d}x+e^{s}\int V(xe^{-s})\, u\, \varphi\, \mathrm{d}x-e^{-s}\int f(e^{2s}u)\, \varphi\, \mathrm{d}x,\label{tju}\\
&\partial_{s}\tilde J(s, u)=s+\frac{3}{2}e^{3s}\int \left[|\nabla u|^{2}+\frac{1}{2}\phi_{u}\, u^{2}\right]\, \mathrm{d}x+\frac{e^{s}}{2}\int \left[V(xe^{-s})-\nabla V(xe^{-s})\cdot xe^{-s}\right]\, u^{2}\, \mathrm{d}x\nonumber\\
&\qquad\qquad\qquad\qquad\qquad\qquad\quad- e^{-3s}\int \left[2f(e^{2s}u)\, e^{2s}u-3F(e^{2s} u)\right]\, \mathrm{d}x.\label{tjs}
\end{align}
\end{proposition}

\begin{proof}
By changing variables, it suffices to prove the statement for $J(t, u)=J(u_{t})$ on $\R_{+}\times X$. A simple scaling argument shows that $\phi_{u_{t}}(x)=t^{2}\phi_{u}(tx)$,
so that the following change of variables is justified by $\phi_{u}\in L^{6}(\R^{3})$ and $u\in L^{12/5}(\R^{3})$
\[
\int \phi_{u_{t}}\, u_{t}^{2}\, \mathrm{d}x=t^{3}\int \phi_{u}\, u^{2}\, \mathrm{d}x.
\]
Similarly,
\begin{align*}
\int |\nabla u_{t}|^{2}\, \mathrm{d}x&=t^{3}\int |\nabla u|^{2}\, \mathrm{d}x\le t^{3}\|u\|^{2},\\
\left|\int F(u_{t})\, \mathrm{d}x\right|&\le Ct\int u^2\,\mathrm{d}x+C|t|^{2p-3}\int |u|^{p}\, \mathrm{d}x.
\end{align*}
For the potential term, thanks to the continuity of $V$ the change of variable $x=y/t$ is justified on any fixed ball $B_{R}$ and  the previous Lemma ensures
\[
\int_{B_{R}} \tilde V\, u_{t}^{2}\, \mathrm{d}x=t\int_{B_{tR}}\tilde V(y/t) u^{2}(y)\, dy\le \max\{t^{\kappa}, t^{-\kappa}\}\int_{B_{tR}}\tilde V\, u^{2}\, dy\le c_{t}\|u\|^{2},
\]
so that letting $R\to +\infty$ proves that $\tilde J$ is well-defined. Formula \eqref{tj} directly follows from changing variables.

Formula \eqref{tju} can be computed in a standard way, while \eqref{tjs} is obtained by deriving under the integral sign in \eqref{tj}. Observe that
\begin{equation}
\label{f1}
\left|2f(e^{2s}u)\, e^{2s}u-3F(e^{2s} u)\right|\le C(e^{4s}|u|^2+e^{2ps}|u|^{6})
\end{equation}
by the growth condition $(f_{1})$ and
\begin{align}
\left|V(xe^{-s})-\nabla V(xe^{-s})\cdot xe^{-s}\right|&\le \left|V(xe^{-s})\right|+\left|\nabla V(x e^{-s}) \cdot xe^{-s}\right|\nonumber\\
&\le m+(1+\kappa)\, \tilde V(xe^{-s}) \le m+(1+\kappa)\, e^{\kappa|s|}\,  \tilde V (x)\label{f2}
\end{align}
due to $|V|\le\tilde{V}+m$, $(V_{2})$ and Lemma \ref{l1}. Moreover both
\[
s\mapsto \int \left|2f(e^{2s}u)\, e^{2s}u-3F(e^{2s} u)\right|\, \mathrm{d}x, \qquad s\mapsto \int \left|V(xe^{-s})-\nabla V(xe^{-s})\cdot xe^{-s}\right| u^{2}\, \mathrm{d}x
\]
are continuous by dominated convergence and  standard arguments yields the differentiation formula \eqref{tjs}. Finally, the  estimates \eqref{f1}, \eqref{f2} ensure the continuity of the corresponding Nemitskii operators appearing in \eqref{tju} and \eqref{tjs}, so that $\tilde J$ is of class $C^{1}$.
\end{proof}

\begin{proposition}[Pohozaev identity]\label{poz}
Assume $(V_{2})$ and $(f_{1})$ and let $u$ be a critical point of $J$ on $X$. Then
\[
\left.\frac{\mathrm{d}}{\mathrm{d}t}\right|_{t=1}J(u_{t})=0.
\]
\end{proposition}

\begin{proof}
Let $u_{(t)}(x)=u(t x)$. The same argument used in the proof of Proposition \ref{proptj}
shows that under assumption $(V_{2})$ the curve $t\mapsto u_{(t)}$ is continuous in $X$ at $t=1$, and the functions $\varphi(t):=J(u_{t})$ and $\psi(t):= J(u_{(t)})$ are differentiable at $t=1$. By the mean value
theorem%
\begin{align}
\nonumber\varphi(t)-\psi(t) &  =J(u_{t})-J(u_{(t)})=\langle DJ(\xi_{t}%
),u_{t}-u_{(t)}\rangle \\
&  =\langle DJ(\xi_{t}),\left(  t^{2}-1\right)  u_{(t)}\rangle
\nonumber\\
&  =(  t^{2}-1)  \langle DJ(\xi_{t}),u_{(t)}\rangle
\label{xmq}
\end{align}
for some $\xi_{t}$ lying on the segment from $u_{t}$ to $u_{(t)}$.
Therefore $\xi_{t}\rightarrow u$ in $X$ as $t\rightarrow1$, because both
$u_{t}$ and $u_{(t)}$ possess this property. Consequently, $DJ(\xi
_{t})\rightarrow DJ(u)=0$ and%
\[
\left\vert \langle DJ(\xi_{t}),u_{(t)}\rangle \right\vert
\leq\left\Vert DJ(\xi_{t})\right\Vert \Vert u_{(t)}\Vert
\rightarrow0\qquad \text{for $t\to 1$}
\]
because $u_{(t)}$ is continuous. It follows from \eqref{xmq} and $\varphi(1)=\psi(1)$ that $\varphi^{\prime
}(1)=\psi^{\prime}(1)$, that is,%
\[
\left.  \frac{\mathrm{d}}{\mathrm{d}t}\right\vert _{t=1}J(u_{t})=\left.
\frac{\mathrm{d}}{\mathrm{d}t}\right\vert _{t=1}J(u_{(t)})\text{.}%
\]
By the usual form of the Pohozaev identity\footnote{This follows from the standard technique (see \cite{MR695535,MR695536}) of multiplying the strong form of the equation by $\nabla u\cdot x$ (notice that  $u\in W^{2,2}_{{\rm loc}}(\R^{3})$ by elliptic regularity), integrate by parts in $B_{R}$ and use the finiteness of the energy  to get rid of the corresponding boundary terms for a suitable sequence of radii $R_{n}\to +\infty$. The nonlocal term involving $\phi_{u}$ can be treated through \cite[Proof of Theorem 1.3]{MR2079817}.} the last term vanishes, thus proving the theorem.
\end{proof}

In the following, we denote by $\tilde{D}\tilde J$ the total differential of $\tilde J$ with respect to both variables $s$ and $u$.
\begin{lemma}\label{crt}
If $(V_{2})$ and $(f_{1})$ hold, then
\[
\tilde D\tilde{J}(\bar s,\bar u)=0\quad \Leftrightarrow \quad \text{$\bar s=0$  \ and \ $DJ(\bar u)=0$}.
\]
\end{lemma}

\begin{proof}
\begin{itemize}
\item[($\Leftarrow$)] From \eqref{tju}, it follows that $DJ(\bar u)=0$ implies $\partial_{u}\tilde J(0, \bar u)=0$. Therefore, it suffices to prove that $\partial_{s}\tilde J(0, \bar u)=0$. From $DJ(\bar u)=0$, Proposition \ref{poz} gives
\begin{equation}
\label{ep}
\left.\frac{\mathrm{d}}{\mathrm{d}t}\right|_{t=1}J(\bar u_{t})=0.
\end{equation}
The map $t\mapsto J(\bar u_{t})$ is $C^{1}$ by Proposition \ref{proptj} and
\[
J(\bar u_{t})=\tilde J(\log t, \bar u)-\frac{\log^{2}t}{2},
\]
so that \eqref{ep} reads
\[
0=\left.\frac{\mathrm{d}}{\mathrm{d}t}\right|_{t=1}\left(\tilde J(\log t, \bar u)-\frac{\log^{2}t}{2}\right)=\left.\left(\partial_{s}\tilde J(\log t, \bar u)\frac{1}{t}-\frac{\log t}{t}\right)\right|_{t=1}=\partial_{s}\tilde J(0, \bar u).
\]
\item[($\Rightarrow$)]From $\tilde D\tilde{J}(\bar s,\bar u)=0$, being $\tilde D=(\partial_{s}, \partial_{u})$, we immediately infer $0=\partial_{u}\tilde{J}(\bar s, \bar u)=\partial_{u} J(\bar u_{e^{\bar s}})$ and we only have to prove that $\bar s=0$. Proposition \ref{poz} applied to $v:=\bar u_{e^{\bar s}}$ gives
\[
\left.\frac{\mathrm{d}}{\mathrm{d}t}\right|_{t=1}J(v_{t})=0.
\]
The function $t\mapsto J(v_{t})$ is $C^{1}$ by Proposition \ref{proptj} and
\[
J(v_{t})=J(\bar u_{te^{\bar s}})=\tilde J(\bar s+\log t, \bar u)-\frac{(\bar s+\log t)^{2}}{2},
\]
where the first equality is due to Remark \ref{rkut}. By the Chain Rule and $\partial_{s}\tilde J(\bar s, \bar u)=0$ we have
\[
0=\left.\frac{\mathrm{d}}{\mathrm{d}t}\right|_{t=1}J(v_{t})=\left.\left(\partial_{s}\tilde J(\bar s+ \log t, \bar u)\frac{1}{t}-\frac{\bar s+\log t}{t}\right)\right|_{t=1}=\partial_{s}\tilde J(\bar s, \bar u)-\bar s=-\bar s.\qedhere
\]
\end{itemize}
\end{proof}

\begin{theorem}\label{PS}
Suppose that $(V_{0})$--$(V_{2})$ and $(f_{1})$--$(f_{2})$ hold. Then $\tilde{J}$ satisfies the $(PS)$ condition.
\end{theorem}

\begin{proof}
Let $\{(s_{n}, u_{n})\}$ be a $(PS)$-sequence for $\tilde J$ in $\R\times X$. Then
\[
|\tilde J(s_{n}, u_{n})|+|\partial_{s}\tilde J(s_{n}, u_{n})|=O(1).
\]
Choose $\lambda\in \left]3, \mu\right[$, where $\mu>3$ is given by $(f_{2})$. Then
\begin{align*}
&(2\lambda-3)\tilde J(s, u)-\partial_{s}\tilde J(s, u)=\frac{2\lambda-3}{2}s^{2}-s+\frac{\lambda-3}{2} e^{3s}\int |\nabla u|^{2}\, \mathrm{d}x\\
&\qquad+\frac{\lambda-3}{2} e^{3s}\int\left[ |\nabla u|^{2}+\phi_{u}\, u^{2}\right]\, \mathrm{d}x +\frac{e^{s}}{2}\int\left[2(\lambda-2)
V(xe^{-s})+\nabla V(xe^{-s})\cdot xe^{-s}\right] u^{2}\, \mathrm{d}x\\
&\qquad+2e^{-3s}\int \left[f(e^{2s}u)\, e^{2s}u-\lambda F(e^{2s}u)\right]\, \mathrm{d}x.
\end{align*}
Using \eqref{RI} on the second integral and $(f_{2})$ on the last, we thus obtain
\begin{align}
&(2\lambda-3)\tilde J(s, u)-\partial_{s}\tilde J(s, u)\ge \frac{2\lambda-3}{2}s^{2}-s+\frac{\lambda-3}{2} e^{3s}\int |\nabla u|^{2}\, \mathrm{d}x+(\lambda-3)e^{3s}\int |u|^{3}\, \mathrm{d}x\nonumber\\
&\qquad+\frac{e^{s}}{2}\int\left[2(\lambda-2)V(xe^{-s})+\nabla V(xe^{-s})\cdot xe^{-s}\right] u^{2}\, \mathrm{d}x+2(\mu-\lambda)e^{-3s}\int F(e^{2s}u)\, \mathrm{d}x.\label{ps1}
\end{align}
The third integral is bounded from below through $(V_{0})$-$(V_{2})$ and H\"older's inequality. Indeed, set
\[
v(x)=e^{3s/2}u(xe^{s}),\qquad W_{\lambda}(x):=2(\lambda-2)V(x)+\nabla V(x)\cdot x.
\]
Then, by a change of variables,
\begin{equation}
\int\left[2(\lambda-2)V(xe^{-s})+\nabla V(xe^{-s})\cdot xe^{-s}\right] u^{2}\, \mathrm{d}x=\int W_{\lambda} \, v^{2}\, \mathrm{d}x.
\label{xxs}
\end{equation}
As $W_\lambda$ is bounded  on bounded sets, we let $C_{\lambda}\in \R$ be such that  $W_{\lambda}\ge -C_{\lambda}$ in $B_{R}$, $R$ given in $(V_{1})$. Then
\begin{equation}
\label{ps6}
\int_{B_{R}} W_{\lambda}\, v^{2}\, \mathrm{d}x\ge -C_{\lambda}\int_{B_{R}} v^{2}\, \mathrm{d}x\ge -C_{\lambda}\left(\int |v|^{3}\, \mathrm{d}x\right)^{\frac 2 3}.
\end{equation}
On $\R^{3}\setminus B_{R}$, we split the integral on the two sets $\{V\ge 0\}$ and $\{V< 0\}$, the latter having finite measure by $(V_{0})$. Because $2(\lambda-2)>2$, assumption $(V_{1})$ implies that
\[
W_{\lambda}(x)=2(\lambda-2)V(x)+\nabla V(x)\cdot x
\ge 2V(x)+\nabla V(x)\cdot x\ge0
\]
for $x\in\{V\ge 0\}\setminus B_{R}$. On the other hand, by $(V_{2})$ and $V\ge -m$ we have
\[
W_{\lambda}\ge 2(\lambda-2)V-\kappa(V+m)\ge -(2(\lambda-2)+\kappa)m \quad \text{on $\{V<0\}$}.
\]
We thus have, for some possibily larger $C_{\lambda}$
\begin{equation}
\label{ps7}
\int_{\R^{3}\setminus B_{R}}W_{\lambda}\, v^{2}\, \mathrm{d}x\ge \int_{\{V<0\}\setminus B_{R}}W_{\lambda}\, v^{2}\, \mathrm{d}x\ge -C_{\lambda}\int_{\{V<0\}} v^{2}\, \mathrm{d}x\ge -C_{\lambda}\, |\{V<0\}|^{\frac 1 3}\left(\int |v|^{3}\, \mathrm{d}x\right)^{\frac 2 3}.
\end{equation}
Combining \eqref{xxs}, \eqref{ps6}, \eqref{ps7} and computing $|v|_{3}$ in terms of $\vert u\vert_{3}$ by changing variable, we get
\begin{equation}
\label{ps2}
\int\left[2(\lambda-2)V(xe^{-s})+\nabla V(xe^{-s})\cdot xe^{-s}\right] u^{2}\, \mathrm{d}x\ge -C_{\lambda}\,  e^{s}\left(\int |u|^{3}\, \mathrm{d}x\right)^{\frac 2 3}.
\end{equation}
Inserting the latter into \eqref{ps1}, for our $(PS)$-sequence $\{(s_n,u_n)\}$, we have
\begin{align*}
O(1)&\ge (2\lambda-3)\tilde J(s_{n}, u_{n})-\partial_{s}\tilde J(s_{n}, u_{n})\\
&\ge
\frac{2\lambda-3}{2}s_{n}^{2}-s_{n}+\frac{\lambda-3}{2} e^{3s_{n}}\int |\nabla u_{n}|^{2}\, \mathrm{d}x+2(\mu-\lambda)e^{-3s_{n}}\int F(e^{2s_{n}}u_{n})\, \mathrm{d}x\\
&\qquad\qquad\qquad+(\lambda-3)e^{3s_{n}}\int |u_{n}|^{3}\, \mathrm{d}x-C_{\lambda}\left(e^{3s_{n}}\int |u_{n}|^{3}\, \mathrm{d}x\right)^{\frac 2 3}.
\end{align*}
From $\lambda>3$ we infer $(\lambda-3)\xi_n-C_{\lambda}\xi_n^{2/3}\to +\infty$ if  $\xi_n=e^{3s_{n}}|u_{n}|_{3}^{3}\to +\infty$, while also using $\mu-\lambda>0$ and $F\ge 0$ we deduce from the previous estimate that
\begin{equation}
\label{ps3}
|s_{n}|,\quad  \int |\nabla u_{n}|^{2}\, \mathrm{d}x,\quad  \int |u_{n}|^{3}\, \mathrm{d}x,\quad \int F(e^{2s_{n}}u_{n})\quad \text{are bounded,}
\end{equation}
and recalling that $\tilde{J}(s_{n}, u_{n})=O(1)$ we also get through the previous bounds
\begin{equation}
\label{ps4}
\int V(xe^{-s_{n}})\, u_{n}^{2}\, \mathrm{d}x\le O(1).
\end{equation}
To complete the proof of the boundedness of $\|u_{n}\|$, let $S\geq1$ be such that $\left\vert \kappa s_{n}\right\vert \leq S$.
Applying Lemma \ref{l1} for $x$ being $xe^{-s_{n}}$ and $t$ being $e^{s_{n}}$, we
get%
\begin{equation}
V(x)+m\leq e^{S}\left(  V(xe^{-s_{n}})+m\right)  \text{.}\label{haide}%
\end{equation}
Choose $k\geq m$ large enough such that%
\[
\frac{1}{2}e^{-S}k\geq\left(  1-e^{-S}\right)  m\text{.}%
\]
Using \eqref{haide}, if $V(x)> k$, we have%
\[
V(xe^{-s_{n}})\geq e^{-S}V(x)-\left(  1-e^{-S}\right)  m\geq\frac{e^{-S}}%
{2}V(x)\text{.}%
\]
Thus, also using $V\geq-m$ on $\left\{  V\le k\right\}  $, we deduce%
\begin{align*}
\int V(xe^{-s_{n}})u_{n}^{2}\,\mathrm{d}x  & = \int_{\left\{
V>k\right\}  }V(xe^{-s_{n}}%
)u_{n}^{2}\,\mathrm{d}x+\int_{\left\{  V\leq k\right\}  }V(xe^{-s_{n}}%
)u_{n}^{2}\,\mathrm{d}x \\
& \geq\frac{e^{-S}}{2}\int_{\left\{  V>k\right\}  }Vu_{n}^{2}\,\mathrm{d}%
x-m\int_{\left\{  V\leq k\right\}  }u_{n}^{2}\,\mathrm{d}x\\
& \geq\frac{e^{-S}}{2}\int_{\left\{  V>k\right\}  }Vu_{n}^{2}\,\mathrm{d}%
x-m\left\vert \left\{  V\leq k\right\}  \right\vert ^{1/3}\left(
\int\left\vert u_{n}\right\vert ^{3}\right)  ^{2/3}\\
& \geq\frac{e^{-S}}{2}\int_{\left\{  V>k\right\}  }Vu_{n}^{2}\,\mathrm{d}%
x-O(1)\text{,}%
\end{align*}
where we used $(V_{0})$ and \eqref{ps3} in the last inequality. From \eqref{ps4} we thus infer
\[
 \int_{\{V>k\}}V\, u_{n}^{2}\, \mathrm{d}x\le O(1).
\]
Finally, due to $k>m$  it holds $V+m\le 2V$ on the set $\{V>k\}$ and $V+m\le 2k$ on $\{V\le k\}$, thus
\begin{align*}
\|u_{n}\|^{2}&\le \int |\nabla u_{n}|^{2}\, \mathrm{d}x+2\int_{\{V>k\}}V\, u_{n}^{2}\, \mathrm{d}x+2k\int_{\{V\le k\}}u_{n}^{2}\, \mathrm{d}x\\
& \le  \int |\nabla u_{n}|^{2}\, \mathrm{d}x+2\int_{\{V>k\}}V\, u_{n}^{2}\, \mathrm{d}x+2k\, |\{V\le k\}|^{\frac 1 3}\left(\int |u_{n}|^{3}\, \mathrm{d}x\right)^{\frac 2 3}\le O(1)
\end{align*}
by \eqref{ps3}, proving the boundedness of $\{u_{n}\}$ in $X$. Finally, the proof of the strong compactness of $\{u_{n}\}$ again follows as in \cite{MR2548724} thanks to the compactness of $X\hookrightarrow L^{p}(\R^{N})$ for $p\in [2, 6[$.
\end{proof}

\begin{lemma}\label{Ml}
Assume $(f_{1})$-$(f_{2})$, $(V_{0})$-$(V_{2})$. For any $\lambda\in \ ]3, \mu[$, there exists $M_{\lambda}\ge 0$ such that
\begin{equation}
\label{Mlemma}
\frac{\mathrm{d}}{\mathrm{d}t} \tilde J(\tau, u_{t})\le \frac{2\lambda-3}{t}\left(\tilde J(\tau, u_{t})+M_{\lambda}\right),\qquad t>0, u\in X, \tau\in \R.
\end{equation}
\end{lemma}

\begin{proof}
The estimate is independent of $\tau$ and $u$, so we let $v=u_{e^{\tau}}$ and observe that
\begin{align}
\label{pxs}\tilde J(\tau, u_{t})=\frac{\tau^{2}}{2}+ J(u_{te^{\tau}})&=\frac{\tau^{2}}{2}-\frac{\log^{2} t}{2}+ \tilde J(\log t, v)\ge \tilde J(\log t, v)-\frac{\log^{2} t}{2}\\
\label{qxs} \frac{\mathrm{d}}{\mathrm{d}t} \tilde J(\tau, u_{t})&=-\frac{\log t}{t}+\frac{1}{t}\partial_{s} \tilde J(\log t, v).
\end{align}
We claim that for given $\lambda\in \ ]3, \mu[$, there exists $M_\lambda>0$ such that
\begin{equation}
\label{el2}
\partial_{s}\tilde J(s, v)-s\le (2\lambda-3)\left(\tilde J (s, v)-\frac{s^{2}}{2}+M_\lambda\right)
\end{equation}
for all $s\in\R$ and $v\in X$. Then, using \eqref{pxs} and \eqref{qxs}, with $s=\log t$ and $v=u_{e^\tau}$ in \eqref{el2} we deduce
\begin{align*}
\frac{\mathrm{d}}{\mathrm{d}t} \tilde J(\tau, u_{t})&=-\frac{\log t}{t}+\frac{1}{t}\partial_{s} \tilde J(\log t, v)\\
&\le\frac{2\lambda-3}{t}\left(\tilde J (\log t, u_{e^\tau})-\frac{\log^{2}t}{2}+M_{\lambda}\right)=\frac{2\lambda-3}{t}\left(J(u_{te^\tau})+M_{\lambda}\right)\\
&\le \frac{2\lambda-3}{t}\left(\tilde J(\tau, u_{t})+M_{\lambda}\right),
\end{align*}
proving \eqref{Mlemma}.
To prove \eqref{el2} ignore the nonnegative terms involving $\nabla u$ and $F$ in \eqref{ps1} and use \eqref{ps2} to get
\[
(2\lambda -3)\tilde J(s, v)-\partial_{s}\tilde J(s, v)\ge\frac{2\lambda -3}{2}s^{2} -s+(\lambda-3)\, e^{3s}\int |v|^{3}\, \mathrm{d}x-C_{\lambda}\left(e^{3s}\int |v|^{3}\, \mathrm{d}x\right)^{\frac 2 3}.
\]
The last two terms are bounded from below thanks to $\lambda>3$, thus \eqref{el2} is proved.
\end{proof}

\begin{lemma}
Suppose $(f_{1})$-$(f_{2})$ and $(V_{1})$ hold true. Then, for any $(s, u)\in \R\times X\setminus\{0\}$ it holds
\[
\lim_{t\to+\infty} \tilde J(s, u_{t})=-\infty.
\]
\end{lemma}

\begin{proof}
Considering $v=u_{e^{s}}$ it suffices to prove that $J(v_{t})\to -\infty$ as $t\to +\infty$. As in the proof of \eqref{tj} we get
\[
J(v_{t})=\frac{t^{3}}{2}\int\left[|\nabla v|^{2}+\frac12\phi_{v}\, v^{2}\right]\, \mathrm{d}x+\frac{t}{2}\int V\, t^{3}v^{2}(tx)\, \mathrm{d}x-t^{-3}\int F(t^{2}v)\, \mathrm{d}x.
\]
Since $v\ne 0$, we can suppose that for some $\eps>0$, $|\{|v|\ge \eps\}|$ is finite and positive and by \eqref{mu0} we have
\[
\int F(t^{2}v)\, \mathrm{d}x\ge C\int_{\{|v|\ge \eps\}} t^{2\mu}\, |v|^{\mu}\, \mathrm{d}x\ge C\, \eps^{\mu}\, |\{|v|\ge \eps\}|\, t^{2\mu}=:C_{v}\, t^{2\mu},
\]
for some $C_{v}>0$ and $t^{2}\ge 1/\eps$.
$V$ is bounded on $B_{R}$, therefore
\[
\int_{B_{R}} V\, t^{3}v^{2}(tx)\, \mathrm{d}x\le \|V\|_{L^{\infty}(B_{R})}\int t^{3}v^{2}(tx)\, \mathrm{d}x=\|V\|_{L^{\infty}(B_{R})}\int v^{2}\, \mathrm{d}x.
\]
Assumption $(V_{1})$ implies that for any $|\omega|=R$ and $r\ge 1$
\[
\frac{\mathrm{d}}{\mathrm{d}r}\left(\tilde V(r\omega) r^{2}\right)=r\left(2\tilde V(r\omega)+\nabla V(r\omega)\cdot r\omega\right)\ge 0,
\]
so that
\[
H(x)=\tilde V(x)\, |x|^{2}\, \chi_{\R^{3}\setminus B_{R}}(x)\quad \text{is radially non-decreasing}.
\]
Letting $w(x)=v (x)/|x|$, we  have
\[
\int_{\R^{3}\setminus B_{R}} \tilde V\, t^{3}v^{2}(tx)\, \mathrm{d}x=t^{2}\int H\, t^{3} w^{2}(tx)\, \mathrm{d}x=t^{2}\int H(x/t)\, w^{2}\, \mathrm{d}x,
\]
and by the monotonicity of $H$, $H(x/t)\, w^{2}\searrow 0$ as $t\to +\infty$, so that by monotone convergence the last integral vanishes as $t\to +\infty$. Therefore
\[
\int V\, t^{3}v^{2}(tx)\, \mathrm{d}x\le \int_{B_{R}} V\, t^{3}v^{2}(tx)\, \mathrm{d}x+\int_{\R^3\setminus B_{R}} \tilde V\, t^{3}v^{2}(tx)\, \mathrm{d}x\le  \|V\|_{L^{\infty}(B_{R})}\int v^{2}\, \mathrm{d}x+o(t^{2}).
\]
Summing up,
\[
J(v_{t})\le\frac{t^{3}}2\int\left[|\nabla v|^{2}+\phi_{v}\, v^{2}\right]\, \mathrm{d}x+\frac{t\,\|V\|_{L^{\infty}(B_{R})}}{2}\int v^{2}\, \mathrm{d}x+o(t^{3})-C_{v}\, t^{2\mu-3}\to -\infty,
\]
as $t\to+\infty$, because  $2\mu-3>3$.
\end{proof}

\begin{theorem}\label{H}
Assume $(f_{1})$-$(f_{2})$ and $(V_{0})$-$(V_{2})$. Then for any sufficiently negative $a\in\R$, it holds
\[
H_{q}(\R\times X, \{\tilde J\le a\})=0\qquad\text{for any $q\in\{0,1,2,\ldots,\}$.}
\]

\end{theorem}

\begin{proof}
Let $\dot X=X\setminus\{0\}$ and consider the continuous map
\[
\R\times \dot X\times \R_{+}\ni (s, u, t)\mapsto (s, u_{t})\in \R\times \dot X.
\]
Fix $\lambda\in \ ]3, \mu[$ and $a<-M_{\lambda}$, where $M_{\lambda}\ge 0$ is given in Lemma \ref{Ml}. Then, by the previous Lemma, for any $(s, u)\in \R\times \dot X$
\[
\lim_{t\to +\infty}\tilde J(s, u_{t})=-\infty.
\]
Therefore, we infer from \eqref{Mlemma} that the implicit equation $\tilde J(s, u_{t})=a$
has a unique solution $t=\varphi(s, u)$ for any $(s, u)\in \R\times \dot X$ such that $\tilde J(s, u)>a$, and $\varphi:\{(s, u)\in \R\times \dot X: \tilde J(s, u)>a\}\to \R_{+}$
is continuous by a standard application of the implicit function theorem. The  map
\[
\Phi:[0, 1]\times \R\times \dot X\to \R\times \dot X,\qquad
 \Phi\left(\xi, (s, u)\right)=
 \begin{cases}
 \left(s, u_{1-\xi+\xi\varphi(s, u)}\right)&\text{if $\tilde J(s, u)>a$},\\
 (s, u)&\text{if $\tilde J(s, u)\le a$},
 \end{cases}
 \]
 is a deformation retract of $ \R\times \dot X$ onto $\{\tilde J \le a\}$, so that by homotopy invariance
 \[
 H_{*}(\R\times X, \{\tilde J\le a\})=H_{*}(\R\times X, \R\times \dot X).
 \]
 Since $\R\times \dot X$ deformation retracts to $\{0\}\times S^{\infty}$, which is contractible in itself, we get the claim.
\end{proof}

\noindent
Recall that  $X_{+}$, $X_{-}$ and $X_{0}$ are the negative, positive and null eigenspaces of the bilinear form defined in \eqref{defQ}.

\begin{lemma}
Assume $(f_{1})$-$(f_{2})$, $(V_{0})$-$(V_{2})$ and consider the decomposition $\R\times X=\tilde X_{-}\oplus \tilde X_{+}$ where
\[
\tilde X_{-}=X_{-}\oplus X_{0},\qquad \tilde X_{+}=\R\oplus X_{+}.
\]
Then the functional $\tilde J$ has a local linking in the following cases
\begin{enumerate}
\item
$\dim  X_{-}>0$, $\dim X_{0}=0$.
\item
$\dim X_{0}>0$ and $F(t)\ge c\, |t|^{\nu}$ for some $\nu<4$.
\end{enumerate}
\end{lemma}

\begin{proof}
We first show that for some $r>0$
\begin{equation}
\label{link1}
\tilde{J}(s,u)>0\qquad\qquad \text{for $\left\vert s\right\vert
<r$, \quad $u\in  X_{+}$, \quad $\left\Vert u\right\Vert <r$.}
\end{equation}
We have
\begin{align*}
\int V(xe^{-s})\, u^{2}\, \mathrm{d}x &  =\int V\, u^{2}\, \mathrm{d}x+\int  \left( \int_{0}^{1}\frac{\mathrm{d}}{\mathrm{d}\tau}
V(e^{-s\tau}x)\, \mathrm{d}\tau\right) u^{2}\, \mathrm{d}x\\
&  =\int V\, u^{2}\, \mathrm{d}x-s\int \left(\int_{0}^{1}\nabla V(e^{-s\tau
}x)\cdot (e^{-s\tau}x )\, \mathrm{d}\tau\right) u^{2}\, \mathrm{d}x\text{,}
\end{align*}
therefore
\begin{align*}
\left\vert \int V(xe^{-s})\, u^{2}\, \mathrm{d}x-\int V\, u^{2}\, \mathrm{d}x\right\vert  &  \le\left\vert s\right\vert \int  \left(\int_{0}^{1}\left\vert \nabla V(e^{-s\tau}x)\cdot (e^{-s\tau}x)  \right\vert\,
\mathrm{d}\tau  \right)u^{2}\, \mathrm{d}x\\
& \le \kappa \left\vert s\right\vert \int \left(\int_{0}^{1}\tilde
{V}(e^{-s\tau}x)\, \mathrm{d}\tau\right)  u^{2}\, \mathrm{d}x\tag*{by $(V_{2})$}\\
&\le\kappa\left\vert s\right\vert \int \left(\int_{0}^{1}e^{\kappa |s|\tau
}\tilde{V}(x)\, \mathrm{d}\tau\right) u^{2}\, \mathrm{d}x\tag*{by \eqref{el1}}\\
& \le \left(  e^{\kappa |s|}-1\right)  \int\tilde
{V}\, u^{2}\, \mathrm{d}x\le O(s)\left\Vert u\right\Vert
^{2}=o (\left\Vert \left(  s,u\right)  \right\Vert ^{2})
\end{align*}
since, as $\left(  s,b\right)  \rightarrow (0, 0)$, $sb^{2}=o(|(s,b)|^{2})$.
Moreover, $(f_{1})$ and \eqref{mu0} imply that
\[
|F(t)|\le C\left(|t|^{\mu}+|t|^{p+1}\right),
\]
so that, for $\|(s,u)\|\to 0$,  we have
\[
e^{-3s}\left\vert\int F(e^{2s}u)\, \mathrm{d}x\right\vert\le C\int |u|^{\mu}+|u|^{p+1}\, \mathrm{d}x\le o(  \left\Vert u\right\Vert ^{2}),
\]
while by \eqref{em} we easily have $|V|\le2\tilde{V}$, hence
\[
|s|\int |V| \, u^{2}\, \mathrm{d}x\le 2|s|\int \tilde V\, u^{2}\, \mathrm{d}x=o(\|(s, u)\|^{2}).
\]
Gathering these estimates and recalling \eqref{ct}, we get
\begin{align*}
\tilde{J}(s,u) &   =\frac{s^{2}}{2}+\frac{e^{3s}}{2}\int\Big[\left\vert \nabla u\right\vert ^{2}+\frac
{1}{2}\phi_{u}\, u^{2}\Big]\, \mathrm{d}x  +\frac{e^{s}}{2}\int V(xe^{-s})\, u^{2}\, \mathrm{d}x-e^{-3s}
\int F(e^{2s}u)\, \mathrm{d}x\\
&  =\frac{s^{2}}{2} + \frac{1+O(s)}{2}\int\left\vert \nabla u\right\vert
^{2}\, \mathrm{d}x  +\frac{1+O(s)}{2}
\int V\, u^{2}\, \mathrm{d}x\\
&\qquad\qquad\qquad\qquad\qquad\qquad\quad+ O(\left\Vert u\right\Vert ^{4})+o(  \left\Vert \left(  s,u\right)  \right\Vert
^{2}) +o(  \left\Vert u\right\Vert^{2})  \\
&  =\frac{1}{2}\int \left[\left\vert \nabla u\right\vert ^{2}+V\, u^{2}\right]
\, \mathrm{d}x  +\frac{s^{2}}{2}+O(  \left\Vert u\right\Vert ^{4})
+o(  \left\Vert \left(  s,u\right)  \right\Vert ^{2}).
\end{align*}
The latter readily yelds \eqref{link1} and, for $s=0$,
\[
\tilde J(0, u)<0\qquad \qquad \text{for $u\in X_{-}$,\quad $\|u\|<r$},
\]
proving the claimed local linking in case (1).

In  case (2) we proceed as in Theorem \ref{thc}: the previous computations yield
\[
\tilde J(0, u)\le  \frac{1}{2}\int  \left[\left\vert \nabla u\right\vert ^{2}+V\, u^{2}\right]\, \mathrm{d}x  +O(  \left\Vert u\right\Vert ^{4}) -c\int |u|^{\nu} \, \mathrm{d}x,
\]
and being all norms in $X_{-}\oplus X_{0}$ equivalent we deduce
\[
\tilde J(0, u)\le -\lambda_{-}\|u_{-}\|^{2}+O(  \left\Vert u\right\Vert ^{4})-c'\, \|u\|^{\nu}\qquad \text{for $u\in X_{-}\oplus X_{0}$}.
\]
Thanks to $\nu<4$, we infer
\[
\tilde J(0, u)<0\qquad \qquad \text{for $u\in X_{-}\oplus X_{0}$,\quad $\|u\|<r$},
\]
concluding the proof in this case.
\end{proof}

\begin{theorem}
Assume $(f_{1})$-$(f_{2})$, $(V_{0})$-$(V_{2})$ and either
\begin{enumerate}
\item
$\dim  X_{-}>0$, $\dim X_{0}=0$.
\item
$\dim  X_{0}>0$ and $F(t)\ge c\, |t|^{\nu}$ for some $\nu<4$ .
\end{enumerate}
Then problem \eqref{e1} has at least a nontrivial solution.
\end{theorem}

\begin{proof}
Theorems \ref{PS}, \ref{H} and the previous Lemma allow to apply \cite[Corollary 2.3]{MR1749421}, giving a critical point $(\bar s, \bar u)\ne (0, 0)$ for $\tilde J$. But then Lemma \ref{crt} forces $\bar s=0$, $\bar u\ne 0$ and $DJ(\bar u)=0$.
\end{proof}

\section*{References}

\bibliography{ref}

\end{document}